\documentclass{amsart}[35pt]
\usepackage{amsmath, amsthm, amscd, amssymb, amsfonts, inputenc, tikz, comment}
\UseRawInputEncoding

\DeclareMathOperator{\SubF}{SubF}
\DeclareMathOperator{\Fac}{Fac}
\DeclareMathOperator{\Dif}{Dif}

\usetikzlibrary{shapes,arrows}

\usepackage{longtable}
\usepackage{float}
\usepackage{etoolbox}
\usepackage[margin=3cm]{geometry}
\AtBeginEnvironment{tabular}{\tiny}
\usepackage{hyperref}
\usepackage{listings}
\definecolor{codegreen}{rgb}{0,0.6,0}
\definecolor{codegray}{rgb}{0.5,0.5,0.5}
\definecolor{codepurple}{rgb}{0.58,0,0.82}
\definecolor{backcolour}{rgb}{0.95,0.95,0.92}

\newtheorem{theorem}{Theorem}[section]

\newtheorem{cor}[theorem]{Corollary}
\theoremstyle{definition}
\newtheorem{defn}[theorem]{Definition}
\newtheorem{ex}[theorem]{Example}

\theoremstyle{remark}
\newtheorem{rem}[theorem]{Remark}
\numberwithin{equation}{section}
\begin{document}
\title[Computational aspects of subindices and sub-factors of finite groups
]{Computational aspects of subindices and subfactors with characterization of finite index stable groups}
\author{M.H. Hooshmand and M.M. Yousefian Arani}
\address{
Department of Mathematics, Shiraz Branch, Islamic Azad University, Shiraz, Iran }
\email{\tt hadi.hooshmand@gmail.com}
\address{
Department of Pure Mathematics, Faculty of Mathematical Sciences, University of Kashan, Kashan, Iran}

\email{\tt momoeysfn@gmail.com}


\date{}

\begin{abstract}
Recently, sub-indices and sub-factors of groups with connections to number theory, additive combinatorics, and factorization of groups
have been introduced and studied.
Since all group subsets are considered in the theory and there are many basic open problems, conjectures, and questions, their computational aspects are of particular importance.
In this paper, by introducing some computational methods and using theoretical approaches together, we not only solve several
problems but also pave the way to study the topic. As the most important result of the study, we completely characterize finite index stable groups.
\end{abstract}
\maketitle
\section{Introduction}
\noindent
While studying periodic type sets and factors of basic algebraic structures (i.e., magmas,
semigroups, groups, etc.), the first author was guided to new concepts under the title
of sub-factors of groups, sub-indices, and index stability of group subsets.
Also, he arrived at a challenging problem in 2014 (see: mathoverflow.net/questions/155986/factor-subset-of-finite-group)
about factorization of (arbitrary) finite groups
(also see: Kourovka Notebook \cite{Khu}, Vol. 20:  Question 20.37\& 19.35, and \cite{MH1}). The conjecture
says: for every factorization $|G|=ab$ of a finite group $G$,
there exist subsets $A,B$ such that $|A| =a$, $|B|=b$, and $G=AB$.
It has been partially answered in \cite{Berg, Bil, MH1}.
After that, sub-factors of groups are introduced and it is shown
that the concept of index of subgroups can be extended to factors and even arbitrary subsets! (\cite{MH2}).
In \cite{MH3} more studies for sub-indices and sub-factors of finite groups have been done.
Characterization of index stable groups is a challenging problem in the theory with many questions,
conjectures, and research projects.
Answers to all of which seem
 unlikely without the use of computational methods since the theory considers all subsets of a group.
Here, the pervious results together some new computational methods enable us to prove a main theorem that
completely characterizes finite index stable groups. Also,
we solve many problems, answer several questions, and prove some related conjectures by using the
theoretical backgrounds form \cite{MH2,MH3} together with the computational methods.
\section{Sub-factors and sub-indices of group subsets with computational aspects}
Let $A,B$ be subsets of a group $G$. We call the product $AB$ direct, and denote it by $A \cdot B$
if the representation of every
  element of $AB$ by $x=ab$  with $a\in A$, $b\in B$ is unique.
Hence, $G=A \cdot B$ if and only if $G=A B$ and the product $A B$ is direct
(a factorization of $G$ by two subsets, the additive notation is $G=A\dot{+} B$).
Putting
$$\Dif_\ell(A):=A^{-1}A \; , \; \Dif_r(A):=AA^{-1}$$
where  $A^{-1}:=\{ a^{-1}:a\in A\}$, we have $AB=A\cdot B$ if and only if $\Dif_\ell(A)\cap \Dif_r(B)=\{ 1\}$.\\
If $G=A \cdot B$ , then $A$ (resp. $B$) a left (resp. right) factor of $G$ related to $B$ (resp. $A$).
We call $A$ a left
factor of $G$ if and only if $G=A \cdot B$ for some $B\subseteq G$
(equivalently, there exists a right factor $B$ of $G$ relative to $A$).
For example, every subgroup is a left (resp. right) factor
relative to its right (resp. left) transversal, hence it is a two sided factor of $G$.
It is clear that there is a right factor of $G$ relative to $A\subseteq G$
if and only if $A$ is a left factor of $G$. In reference \cite{MH2}, the first author achieved a generalization of factors
that not only does not have the deficiency of factors but also leads to the important concept of sub-indices for all subsets of groups.
\subsection{Sub-factors of groups.}
Let $A$ be a fixed subset of $G$. We call $B$  a right sub-factor of $G$
related to $A$  if $B$ is an inclusion-maximal subset of $G$ with respect to the property $AB=A\cdot B$. Also,
we say $B$ is a right sub-factor of $G$ if it is a right sub-factor related to some subsets
of $G$ (left sub-factors are defined analogously). For example $B=\{2,4\}$ is a sub-factor of
$\mathbb{Z}_5$ (related to $A=\{0,1\}$) but not a factor. Also, all cosets of $B$ are sub-factors of $\mathbb{Z}_5$
(although, $\mathbb{Z}_5$ does not have non-trivial factors). \\
It is proved that every subset of a group has related right and left sub-factors, but it does not need that it
is a sub-factor (e.g., $\mathbb{Z}_+$ is not a sub-factor of the additive group of integers, since $\Dif(\mathbb{Z}_+)=\mathbb{Z}$,
although it is a sub-semigroup).
Also, $B$ is a right sub-factor (of $G$) related to $A$ if and only if
\begin{equation}
\Dif_\ell(A)\cap \Dif_r(B)=\{ 1\}\; , \; \Dif_\ell(A)B=G
\end{equation}
Therefore, we conclude some important results for groups including:\\
(1) For every $A\subseteq G$, the equation $\Dif_\ell(A)X=G$ has some solutions with the condition
$\Dif_\ell(A)\cap \Dif_r(X)=\{ 1\}$.\\
(2) For every $A\subseteq G$, the equation $\Dif_\ell(A)X=G$ has minimal solutions (the largest solution is $G$).\\
(3) The property $AX=A\cdot X$ has maximal solutions (the least solution is $\emptyset$).\\
Putting
$$
\Fac_r(A)=\Fac_r(G:A):=\{ B\subseteq G : B \; \mbox{is a right factor of $G$
related to $A$}\},
$$
$$
\SubF_r(A)=\SubF_r(G:A):=\{ B\subseteq G : B \; \mbox{is a right sub-factor of $G$
related to $A$}\},
$$
we have $\Fac_r(A)\subseteq \SubF_r(A)\neq \emptyset$. But $\Fac_r(A)\neq \emptyset$
if and only if $A$ is a left factor of $G$.\\
\textbf{Computational aspects of sub-factors.}
Since there are many right sub-factors $B$ related to a subset $A$, in practice, we need to know
limitations on such $B$, and applicable methods and algorithms for computing sub-factors (of finite groups).
The following are some useful facts for the computational aspect.\\
\textbf{(a)} It is enough to consider subsets $B$ containing the identity 1. Because putting
$$
\SubF^1_r(A)=\SubF^1_r(G:A):=\{ B_1\in \SubF_r(A): 1\in B_1 \},
$$
(this agrees with the notation $X^1:=X\cup\{1\}$ for every $X\subseteq G$).
We have $B\in \SubF_r(A)$ if and only if $B=B_1\beta$ for some $B_1\in \SubF^1_r(A)$ and $\beta\in B$
(note that $B\neq\emptyset$, and consider $B_1:=Bb_0^{-1}$ for a $b_0\in B$).
Hence
$$\{B_1g: B_1\in\SubF^ 1_r(A) , g\in G\}=\SubF_r(A).$$
\textbf{(b)} Putting $\mathcal{C}_\ell(A):=G\setminus \Dif_\ell(A)$
we conclude that $B_1\subseteq \mathcal{C}_\ell(A)\cup\{1\}=\mathcal{C}^1_\ell(A)$ (since $\Dif_\ell(A)\cap \Dif_r(B_1)=\{ 1\}$
and $B_1\subseteq \Dif_r(B_1)$ if $AB_1=A\cdot B_1$ and $1\in B_1$), for every $B_1\in \SubF^ 1_r(A)$.
Hence, for finite groups $G$ we have
\begin{equation} \label{subf1}
\SubF^1_r(A)\subseteq \Biggl\{\mathbf{B}\subseteq \mathcal{C}^1_\ell(A): 1\in \mathbf{B},
\Bigl\lceil\frac{|G|}{|\Dif_\ell(A)|}\Bigr\rceil\leq |\mathbf{B}|
\leq \Bigl\lfloor \frac{|G|}{|A|}\Bigr\rfloor \Biggr\}
\end{equation}
Therefore it is enough to check only elements of the right hand of (\ref{subf1}) for finding right sub-factors
of $G$ related to $A$, that in this case, the calculations will be much less.
Hence, we can write a GAP code for computing $\SubF_r(A)$ as follows (\href{https://github.com/momoeysfn/Subindices/blob/main/lib/sf.g}{link to code}).
\begin{ex} Consider the additive group $G:=\mathbb{Z}_6$ and $A=\{ 0,1 \}\subseteq G$. By using the above code, we obtain
	 $\mbox{SubF}(A)= \{ \{ 0, 2, 4 \}, \{ 0, 3 \}, \{ 1, 3, 5 \}, \{ 1, 4 \}, \{ 2, 5 \} \} $. Also, see
\href{https://github.com/momoeysfn/Subindices/blob/main/examples/Ex21.md}{(link to more examples)}.
\end{ex}
\textbf{(c)} Another way for computing $\SubF_r(A)$ is applying an algorithm in \cite{MH3}.
Indeed, the relation
\begin{equation}
B\in \SubF_r(A) \; \Leftrightarrow \;
\forall b\in B\; ; \; b\in \bigcap_{\beta\in B\setminus\{ b\}}\mathcal{C}_\ell(A)\beta
\; \& \;  \bigcap_{\beta\in B}\mathcal{C}_\ell(A)\beta=\emptyset.
\end{equation}
had led us to the following theorem.
\begin{theorem}[\cite{MH3}]
\label{rbsfalg}
Let $G$ be a finite group and $A\subseteq G$. Fix $g_0\in G$ and put  $\mathcal{C}^{(-1)}_\ell(A):=G$.
Then, construct the finite sequences $\{g_n\}_{n\geq 0}$ and $\{\mathcal{C}_\ell^{(n)}(A)\}_{n\geq-1}$ as follows:\\
By the assumption $g_0,\cdots , g_{n}$ and $\mathcal{C}^{(-1)}_\ell(A),\cdots , \mathcal{C}_\ell^{(n-1)}(A)$ are defined, for every integer
 $n\geq 0$, set
\begin{equation}
\mathcal{C}_\ell^{(n)}(A):=\mathcal{C}_\ell^{(n-1)}(A)\cap \mathcal{C}_\ell(A)g_{n},
\end{equation}
 and then choose an  element $g_{n+1}$ in
$\mathcal{C}_\ell^{(n)}(A)$ if it is nonempty, and also put $B_n:=\{g_0,\cdots , g_{n+1}\}$
$($thus $B_{-1}=\{g_0\}$, $B_{n-1}\cup \{g_{n+1}\}=B_n$ for all $n\geq 0)$.\\
Then there exists a least integer $N\geq 0$ such that $\mathcal{C}_\ell^{(N)}(A)=\emptyset$, and
$B:=B_{N-1}$ $($with $N+1$ elements$)$ is a right sub-factor of $G$ related to $A$.
\end{theorem}
In \cite{MH3} we introduce a conjecture mentions that every right sub-factor of $G$ related to $A$
can be gotten from the above algorithm. Now, we prove it.
\begin{theorem}
\label{rbsfproof}
Every right sub-factor of $G$ related to $A$ is obtained from the above algorithm $($i.e., the set of all outputs $B$ of
the algorithm is equal to $\SubF_r(A).)$
\end{theorem}
\begin{proof}
Let $X\in\SubF_r(A)$ and represent its members by $X=\{x_0,\cdots , x_{m+1}\}$ where $m=|X|-2$
(thus $m\geq -1$). Now in the algorithm choose $g_0:=x_0$ (since $g_0$ is arbitrary in it).
If $\mathcal{C}_\ell(A)=\emptyset$, then $\Dif_\ell(A)=G$ and hence $|X|=1$, $N=0$ (in the algorithm) and
$X=\{x_0\}=B=B_{-1}$ thus we are done. Otherwise, suppose that $g_0,\cdots,g_n$ take
the values $x_0,\cdots,x_n$, respectively,  for
some $n<m+1$. Since the product $A (\{x_0,\cdots , x_{n}\}\cup \{x_{n+1}\})$ is direct,
$(2.4)$ requires that
$$
x_{n+1}\in \bigcap_{i=0}^n\mathcal{C}_\ell(A)x_i=\bigcap_{i=0}^n\mathcal{C}_\ell(A)g_i=\mathcal{C}_\ell^{(n)}(A)
$$
So $g_{n+1}$ can take the value $x_{n+1}$ in the algorithm process.
Therefore $X\subseteq B_{N-1}$ and so $X=B_{N-1}$ (and $N=m$) since  $X,B_{N-1}\in\SubF_r(A)$.
\end{proof}
By using Theorem \ref{rbsfalg}, \ref{rbsfproof}, we are now enable to write another GAP code for computing the whole $\SubF_r(A)$ as \href{https://github.com/momoeysfn/Subindices/blob/main/lib/bsf.g}{(link to code)}. This also gives us a constructive method to compute an arbitrary sub-factor of $G$ related to $A$ which is much more efficient in larger groups, see \href{https://github.com/momoeysfn/Subindices/blob/main/lib/rbsfrandom.g}{(link to code)}
\begin{ex} Considering
	 $ G:=D_8=\{1,a,a^2,a^3,b, ab, a^2b, a^3b \}, \ A:=\{a,a^2,b\}$, we obtain $$\SubF_r(A)= \Bigl\{ \{ 1, b \}, \{ 1, a^2 \}, \{ 1, ba^3 \}, \{ 1, b \}, \{ b, ba^2 \}, \{ a, b \}, \
	$$ $$\{ a, ba \}, \{ ba, ba^3 \}, \{ a, a^3 \}, \{ a^2, ba^2 \}, \\ \{ a^3, ba^2 \}, \{ a^3, ba^3 \} \Bigr\},$$
by using the second code. Also, see
\href{https://github.com/momoeysfn/Subindices/blob/main/examples/Ex24.md}{(link to more examples)}.
\end{ex}
\subsection{Sub-indices of group subsets.} \label{indsubsets}
For each subset $A$ of a group $G$ we assign sub-indices of $A$ as follows:
$$
|G:A|^+:=\sup\{|B|: B\in \SubF_r(A)\}\; : \; \mbox{right upper index of $A$ (in $G$)};
$$
$$
|G:A|^-:=\inf\{|B|: B\in \SubF_r(A)\}\; : \; \mbox{right lower index of $A$ (in $G$)};
$$
The left notations $|G:A|_\pm$ are defined analogously. Now, we call $A$:\\
(a) right (resp. left) index stable in $G$ if $|G:A|^+=|G:A|^-$ (resp. $|G:A|_+=|G:A|_-$),
and we use the notation $|G:A|_r$ (resp. $|G:A|_\ell$)
for the common value and call it right (resp. left) index of $A$ in $G$. \\
(b) index stable (in $G$) if all of its four sub-indices are equal (equivalently $|G:A|_r=|G:A|_\ell$),
and the common value is denoted by $|G:A|$ and is called the index of $A$ in $G$ (a unique cardinal number corresponding to $A$).\\
Also, a group is called index stable (resp. right index stable) if all its subsets are index stable (resp. right index stable).\\
It is worth noting that if $G$ is a group and $H$ a subgroup, then $H$ (as a subset) is always index stable in $G$, but
as an  independently group, $H$ may be not index stable (i.e., it contains a subset that is not index stable in $H$).\\
The following are some examples of index stable groups and subsets:\\
- Every group of order $<8$ except $C_6$  is index stable.\\
- The only index stable cyclic groups are $C_1$, ..., $C_5$ and $C_7$.\\
- If $A$ is a left (resp. right) difference-generating subset (i.e., $\Dif_\ell(A)=G$), then $A$ and all its upper
subsets are right (resp. left) index stable with the right (resp. left) index $1$.\\
There are some basic properties of sub-indices in arbitrary and finite groups (see \cite{MH2,MH3}).
 The followings are some important results for the finite case:\\
(a)
\begin{equation}
|G:A|^+\leq |G|-|\Dif_\ell(A)|+1,
\end{equation}
and if $A\neq \emptyset$, then
\begin{equation}
\frac{|G|}{|A|^2-|A|+1}\leq
 \lceil\frac{|G|}{|\Dif_\ell(A)|}\rceil\leq
|G:A|^-\leq |G:A|^+ \leq \lfloor \frac{|G|}{|A|}\rfloor \leq \frac{|\Dif_\ell(A)|}{|A|}|G:A|^-
\end{equation}
(b)
If $A\subseteq H\leq G$ and $A$ is right index stable in $G$, then it is so in $H$ and we have
 $$|G:A|_r=|G:H||H:A|_r.$$
Therefore, if $G$  is index stable, then every $H\leq G$ is
so, and $|G:A|=|G:H||H:A|$, for all $A\subseteq H$. \\
(c) Every finite group containing a non-index stable subgroup is non-index stable.\\

There is an important property for sub-indices of finite group subsets that
if $|A|>\frac{|G|}{2}$, then $Dif_\ell(A)=Dif_r(A)=G$, $A$ is index stable, and so $|G:A|=1$. The converse is not true (e.g.,
if $G:=\mathbb{Z}_6, A:=\{0,1,3\}$, then $|G:A|=2$). But as a weak converse, if $|G:A|=1$, then
the inequality $(2.6)$ implies that $\frac{|G|}{|A|^2-|A|+1}\leq 1$, and so
$|A|\geq \frac{1}{2}+\sqrt{|G|-\frac{3}{4}}> \sqrt{|G|}$ (if $G\neq 1$). We will state some counter examples for
related questions in the next section.
In view of this fact, for computational aspects of sub-indices, one may consider the partition
$\{\mathcal{A}_r\}_{r=1}^{|G|}$ for $2^G\setminus\{\emptyset\}$  where
$$
\mathcal{A}_r:=\mathcal{A}_r(G):=\{A\subseteq G: \frac{|G|}{r+1}<|A|\leq \frac{|G|}{r}\}
$$
All sub-indices of every element of $\mathcal{A}_r$ are $\leq r$ (because $\frac{|G|}{r+1}<|A|\leq \frac{|G|}{r}$
if and only if $\lfloor\frac{|G|}{|A|}\rfloor=r$).\\
As a weak converse, let $r$ be an integer such that $1\leq r\leq |G|$. If one of the
sub-indices of $A$ is $\leq r$, then
$$|A|\geq \frac{1}{2}+\sqrt{\frac{|G|}{r}-\frac{3}{4}}> \sqrt{\frac{|G|}{r}}.$$
Hence, $\sqrt{\frac{|G|}{r}}<|A|\leq \frac{|G|}{r}$ if $r\leq \frac{|G|}{|A|}$ (but we can not conclude that $A\in \mathcal{A}_r(G)$).
Since for $r>\frac{|G|}{2}$ and $r=1$ all elements of $\mathcal{A}_r(G)$ are index stable,
it is enough to study $\mathcal{A}_r(G)$ for $2\leq r\leq \frac{|G|}{2}$ (i.e., $\{\mathcal{A}_r\}_{r=2}^{\lfloor\frac{|G|}{2}\rfloor}$).
It is worth noting that all elements of $\mathcal{A}_2(G)$ are left and right index stable with sub-indices 1 or 2,
and they are index stable if $G$ is abelian (see Corollary 3.4 of \cite{MH3}).\\

\textbf{Computational aspects of sub-indices.}
To calculate $|G:A|^+$, the straightforward way is computing $\SubF^1_r(A)$, and then
the maximum of sizes of its elements.
But with a closer look, it can be seen that there is another algorithm, since according to (\ref{subf1})
it is enough to do the next steps:\\
(1) Start with subsets $B\ni 1$ of $\mathcal{C}_\ell(A)\cup\{1\}$ of sizes $\Bigl\lfloor \frac{|G|}{|A|}\Bigr\rfloor$, and then all subsets
of sizes $\Bigl\lfloor \frac{|G|}{|A|}\Bigr\rfloor-1$, and so on.\\
(2) Find the first such $B$ for which $B\in \SubF^1_r(A)$ and denote it by $B_0$ .\\
(3) $|G:A|^+=|B_0|$.\\
Analogous algorithm exists for computing $|G:A|^-$ (and other sub-indices).\\
We are now enable to write an appropriate GAP code for computing the right sub-indices and checking right index stability
 of subsets as \href{https://github.com/momoeysfn/Subindices/blob/main/lib/id.g}{(link to code)}.
\begin{ex}
If $G:=S_3, \ A:=\{(),(1,3,2)\}$, then $A$ is (right and left) index stable and $|G:A|=2$. Also, see
\href{https://github.com/momoeysfn/Subindices/blob/main/examples/Ex25.md}{(link to more examples)}.
\end{ex}

\subsection{A table of sub-indices for $k$-index stability of groups of small orders}
In the theory of sub-indices, we observe that the cardinality of subsets plays an important role for
index stability. Hence we recall a definition from \cite{MH2, MH3}.
\begin{defn}
Let $G$ be a finite group and $1\leq k\leq |G|$ a given integre number. We call
$G$ $k$-index stable if all its subsets of size $k$ are index stable (analogously
for left and right $k$-index stabilities). Also, we convent that $G$ is
$\kappa$-index stable for every $\kappa>|G|$.
\end{defn}
Note that a group is right $k$-index (resp. index) stable if and only if it is left $k$-index (resp. index) stable, since
$$
|G:A^{-1}|^+=|G:A|_+\; , \; |G:A^{-1}|^-=|G:A|_-\; , \; |A^{-1}|=|A|,
$$
for all $A\subseteq G$ (see Theorem 3.12(c) of \cite{MH2}).
For finite groups, we prove in the next section that right, left, and two-sided index stability of finite groups
are equivalent but this is not true for $k$-index stability (for the first counterexample, $A_4$, $D_{12}$ are right and left 6-index stable but not two-sided 6-index stable).
Now, using the Gap code \href{https://github.com/momoeysfn/Subindices/blob/main/examples/Table.g}{(link to code)} which is obtained according to the stated facts and algorithms,
we present a complete table for right and two-sided $k$-index (and index) stability of finite groups of orders$\leq 27$.
It is worth noting that many cases of the table have also theoretical evidence in \cite{MH2, MH3}.
Note that in the following table, there are columns that indicate the state of the right and
 two-sided index stability of subsets of the mentioned size with an ordered pair of 0's and 1's.
 The first component of the ordered couple corresponds to the \underline{$k$-right index stability}
  and the next one to the \underline{$k$-index stability}, where the number zero means
   that it is not established and one indicates that the related property is satisfied.
   Hence, the second component is less than or equal to the first one. For example,
   the column 6 for $A_4$ indicates that $A_4$ is right 6-index stable but not two-sided 6-index stable. Note that since every group of order$\leq 5$ is index stable and all finite groups $G$ are $k$-index stable for $k> \lfloor \frac{|G|}{2}\rfloor $, we do not mention these cases in the table.
Also, notice that every $A\in \mathcal{A}_2(G)$ is right index stable (but not necessarily index stable), thus the first component in the $k$th column
is 1 for all $\lceil\frac{|G|}{3}\rceil+1\leq k \leq  \lfloor \frac{|G|}{2}\rfloor$.

  \begin{longtable}{|l|l|l|l|l|l|l|l|l|} \caption{$k$-index stability of groups; $6 \leq |G| \leq 16$} \label{table} \\
  	\hline
  	\textbf{Group} &
  	\textbf{k=2} &
  	\textbf{k=3} &
  	\textbf{k=4} &
  	\textbf{k=5} &
  	\textbf{k=6} &
  	\textbf{k=7} &
  	\textbf{k=8} &
  	\textbf{\begin{tabular}[c]{@{}l@{}} (right) index stability\end{tabular}} \\ \hline
  	\endhead
  	{\color[HTML]{009901} $S_3$} &
  	{\color[HTML]{009901} 1-1} &
  	{\color[HTML]{009901} 1-1} &
  	&
  	&
  	&
  	&
  	&
  	{\color[HTML]{009901} Index Stable} \\ \hline
  	$C_6$ &
  	{\color[HTML]{FE0000} 0-0} &
  	1-1 &
  	&
  	&
  	&
  	&
  	&
  	None (right) index stable \\ \hline
  	{\color[HTML]{009901} $C_7$} &
  	{\color[HTML]{009901} 1-1} &
  	{\color[HTML]{009901} 1-1} &
  	&
  	&
  	&
  	&
  	&
  	{\color[HTML]{009901} Index Stable} \\ \hline
  	$C_8$ &
  	{\color[HTML]{FE0000} 0-0} &
  	1-1 &
  	1-1 &
  	&
  	&
  	&
  	&
  	None (right) index stable \\ \hline
  	{\color[HTML]{009901} $C_4 \times C_2$} &
  	{\color[HTML]{009901} 1-1} &
  	{\color[HTML]{009901} 1-1} &
  	{\color[HTML]{009901} 1-1} &
  	{\color[HTML]{009901} } &
  	&
  	&
  	&
  	{\color[HTML]{009901} Index Stable} \\ \hline
  	{\color[HTML]{009901} $D_8$} &
  	{\color[HTML]{009901} 1-1} &
  	{\color[HTML]{009901} 1-1} &
  	{\color[HTML]{009901} 1-1} &
  	{\color[HTML]{009901} } &
  	&
  	&
  	&
  	{\color[HTML]{009901} Index Stable} \\ \hline
  	{\color[HTML]{009901} $Q_8$} &
  	{\color[HTML]{009901} 1-1} &
  	{\color[HTML]{009901} 1-1} &
  	{\color[HTML]{009901} 1-1} &
  	{\color[HTML]{009901} } &
  	&
  	&
  	&
  	{\color[HTML]{009901} Index Stable} \\ \hline
  	{\color[HTML]{009901} $C_2 \times C_2 \times C_2$} &
  	{\color[HTML]{009901} 1-1} &
  	{\color[HTML]{009901} 1-1} &
  	{\color[HTML]{009901} 1-1} &
  	{\color[HTML]{009901} } &
  	&
  	&
  	&
  	{\color[HTML]{009901} Index Stable} \\ \hline
  	$C_9$ &
  	{\color[HTML]{FE0000} 0-0} &
  	{\color[HTML]{FE0000} 0-0} &
  	1-1 &
  	&
  	&
  	&
  	&
  	None (right) index stable \\ \hline
  	{\color[HTML]{009901} $C_3 \times C_3$} &
  	{\color[HTML]{009901} 1-1} &
  	{\color[HTML]{009901} 1-1} &
  	{\color[HTML]{009901} 1-1} &
  	{\color[HTML]{009901} } &
  	&
  	&
  	&
  	{\color[HTML]{009901} Index Stable} \\ \hline
  	$D_{10}$ &
  	1-1 &
  	{\color[HTML]{FE0000} 0-0} &
  	1-1 &
  	1-1 &
  	&
  	&
  	&
  	{\color[HTML]{333333} None (right) index stable} \\ \hline
  	$C_{10}$ &
  	{\color[HTML]{FE0000} 0-0} &
  	{\color[HTML]{FE0000} 0-0} &
  	1-1 &
  	1-1 &
  	&
  	&
  	&
  	None (right) index stable \\ \hline
  	$C_{11}$ &
  	{\color[HTML]{FE0000} 0-0} &
  	1-1 &
  	1-1 &
  	1-1 &
  	&
  	&
  	&
  	{\color[HTML]{333333} None (right) index stable} \\ \hline
  	$C_3 : C_4$ &
  	{\color[HTML]{FE0000} 0-0} &
  	{\color[HTML]{FE0000} 0-0} &
  	{\color[HTML]{FE0000} 0-0} &
  	1-1 &
  	1-1 &
  	&
  	&
  	{\color[HTML]{333333} None (right) index stable} \\ \hline
  	$C_{12}$ &
  	{\color[HTML]{FE0000} 0-0} &
  	{\color[HTML]{FE0000} 0-0} &
  	{\color[HTML]{FE0000} 0-0} &
  	1-1 &
  	1-1 &
  	&
  	&
  	{\color[HTML]{333333} None (right) index stable} \\ \hline
  	$A_4$ &
  	1-1 &
  	{\color[HTML]{FE0000} 0-0} &
  	{\color[HTML]{FE0000} 0-0} &
  	1-1 &
  	1-0 &
  	&
  	&
  	{\color[HTML]{333333} None (right) index stable} \\ \hline
  	$D_{12}$ &
  	{\color[HTML]{FE0000} 0-0} &
  	{\color[HTML]{FE0000} 0-0} &
  	{\color[HTML]{FE0000} 0-0} &
  	1-1 &
  	1-0 &
  	&
  	&
  	{\color[HTML]{333333} None (right) index stable} \\ \hline
  	$C_6 \times C_2$ &
  	{\color[HTML]{FE0000} 0-0} &
  	{\color[HTML]{FE0000} 0-0} &
  	{\color[HTML]{FE0000} 0-0} &
  	1-1 &
  	1-1 &
  	&
  	&
  	{\color[HTML]{333333} None (right) index stable} \\ \hline
  	$C_{13}$ &
  	{\color[HTML]{FE0000} 0-0} &
  	{\color[HTML]{FE0000} 0-0} &
  	{\color[HTML]{FE0000} 0-0} &
  	1-1 &
  	1-1 &
  	&
  	&
  	{\color[HTML]{333333} None (right) index stable} \\ \hline
  	$D_{14}$ &
  	1-1 &
  	{\color[HTML]{FE0000} 0-0} &
  	{\color[HTML]{FE0000} 0-0} &
  	1-1 &
  	1-0 &
  	1-0 &
  	&
  	{\color[HTML]{333333} None (right) index stable} \\ \hline
  	$C_{14}$ &
  	{\color[HTML]{FE0000} 0-0} &
  	{\color[HTML]{FE0000} 0-0} &
  	{\color[HTML]{FE0000} 0-0} &
  	1-1 &
  	1-1 &
  	1-1 &
  	&
  	{\color[HTML]{333333} None (right) index stable} \\ \hline
  	$C_{15}$ &
  	{\color[HTML]{FE0000} 0-0} &
  	{\color[HTML]{FE0000} 0-0} &
  	{\color[HTML]{FE0000} 0-0} &
  	{\color[HTML]{FE0000} 0-0} &
  	1-1 &
  	1-1 &
  	&
  	{\color[HTML]{333333} None (right) index stable} \\ \hline
  	$C_{16}$ &
  	{\color[HTML]{FE0000} 0-0} &
  	{\color[HTML]{FE0000} 0-0} &
  	{\color[HTML]{FE0000} 0-0} &
  	{\color[HTML]{FE0000} 0-0} &
  	1-1 &
  	1-1 &
  	1-1 &
  	{\color[HTML]{333333} None (right) index stable} \\ \hline
  	$C_4 \times C_4$ &
  	1-1 &
  	{\color[HTML]{FE0000} 0-0} &
  	{\color[HTML]{FE0000} 0-0} &
  	1-1 &
  	1-1 &
  	1-1 &
  	1-1 &
  	None (right) index stable \\ \hline
  	$(C_4 \times C_2) : C_2$ &
  	1-1 &
  	1-1 &
  	{\color[HTML]{FE0000} 0-0} &
  	1-1 &
  	1-0 &
  	1-0 &
  	1-0 &
  	{\color[HTML]{333333} None (right) index stable} \\ \hline
  	$C_4 : C_4$ &
  	1-1 &
  	1-1 &
  	{\color[HTML]{FE0000} 0-0} &
  	1-1 &
  	1-1 &
  	1-1 &
  	1-1 &
  	{\color[HTML]{333333} None (right) index stable} \\ \hline
  	$C_8 \times C_2$ &
  	{\color[HTML]{FE0000} 0-0} &
  	{\color[HTML]{FE0000} 0-0} &
  	{\color[HTML]{FE0000} 0-0} &
  	1-1 &
  	1-1 &
  	1-1 &
  	1-1 &
  	{\color[HTML]{333333} None (right) index stable} \\ \hline
  	$C_8 : C_2$ &
  	{\color[HTML]{FE0000} 0-0} &
  	{\color[HTML]{FE0000} 0-0} &
  	{\color[HTML]{FE0000} 0-0} &
  	1-1 &
  	1-0 &
  	1-0 &
  	1-0 &
  	{\color[HTML]{333333} None (right) index stable} \\ \hline
  	$D_{16}$ &
  	{\color[HTML]{FE0000} 0-0} &
  	{\color[HTML]{FE0000} 0-0} &
  	{\color[HTML]{FE0000} 0-0} &
  	{\color[HTML]{FE0000} 0-0} &
  	1-1 &
  	1-0 &
  	1-0 &
  	{\color[HTML]{333333} None (right) index stable} \\ \hline
  	${QD}_{16}$ &
  	{\color[HTML]{FE0000} 0-0} &
  	1-1 &
  	{\color[HTML]{FE0000} 0-0} &
  	1-1 &
  	1-0 &
  	1-0 &
  	1-0 &
  	None (right) index stable \\ \hline
  	$Q_{16}$ &
  	{\color[HTML]{FE0000} 0-0} &
  	{\color[HTML]{FE0000} 0-0} &
  	{\color[HTML]{FE0000} 0-0} &
  	1-1 &
  	1-1 &
  	1-1 &
  	1-1 &
  	{\color[HTML]{333333} None (right) index stable} \\ \hline
  	$C_4 \times C_2 \times C_2$ &
  	1-1 &
  	1-1 &
  	{\color[HTML]{FE0000} 0-0} &
  	1-1 &
  	1-1 &
  	1-1 &
  	1-1 &
  	{\color[HTML]{333333} None (right) index stable} \\ \hline
  	$C_2 \times D_8$ &
  	1-1 &
  	1-1 &
  	{\color[HTML]{FE0000} 0-0} &
  	1-1 &
  	1-1 &
  	1-0 &
  	1-0 &
  	{\color[HTML]{333333} None (right) index stable} \\ \hline
  	$C_2 \times Q_8$ &
  	1-1 &
  	1-1 &
  	{\color[HTML]{FE0000} 0-0} &
  	1-1 &
  	1-1 &
  	1-1 &
  	1-1 &
  	{\color[HTML]{333333} None (right) index stable} \\ \hline
  	$(C_4 \times C_2) : C_2$ &
  	1-1 &
  	1-1 &
  	{\color[HTML]{FE0000} 0-0} &
  	1-1 &
  	1-1 &
  	{\color[HTML]{FE0000} 1-0} &
  	1-0 &
  	{\color[HTML]{333333} None (right) index stable} \\ \hline
  	{\color[HTML]{32CB00} $C_2 \times C_2 \times C_2 \times C_2$} &
  	{\color[HTML]{32CB00} 1-1} &
  	{\color[HTML]{32CB00} 1-1} &
  	{\color[HTML]{32CB00} 1-1} &
  	{\color[HTML]{32CB00} 1-1} &
  	{\color[HTML]{32CB00} 1-1} &
  	{\color[HTML]{32CB00} 1-1} &
  	{\color[HTML]{32CB00} 1-1} &
  	{\color[HTML]{009901} Index Stable} \\ \hline
  \end{longtable}

  \begin{longtable}[c]{|l|l|l|l|l|l|l|l|l|l|l|l|l|l|} \caption{$k$-index stability of groups; $17 \leq |G| \leq 27$} \\
  	\hline
  	\textbf{Group} &
  	\textbf{2} &
  	\textbf{3} &
  	\textbf{4} &
  	\textbf{5} &
  	\textbf{6} &
  	\textbf{7} &
  	\textbf{8} &
  	\textbf{9} &
  	\textbf{10} &
  	\textbf{11} &
  	\textbf{12} &
  	\textbf{13} &
  	\textbf{} \\ \hline
  	\endhead
  	$C_{17}$ &
  	{\color[HTML]{FE0000} 0-0} &
  	{\color[HTML]{FE0000} 0-0} &
  	{\color[HTML]{FE0000} 0-0} &
  	{\color[HTML]{FE0000} 0-0} &
  	1-1 &
  	1-1 &
  	1-1 &
  	&
  	&
  	&
  	&
  	&
  	\begin{tabular}[c]{@{}l@{}}None (right)\\ index stable\end{tabular} \\ \hline
  	$D_{18}$ &
  	{\color[HTML]{FE0000} 0-0} &
  	{\color[HTML]{FE0000} 0-0} &
  	{\color[HTML]{FE0000} 0-0} &
  	{\color[HTML]{FE0000} 0-0} &
  	{\color[HTML]{FE0000} 0-0} &
  	1-0 &
  	1-0 &
  	1-0 &
  	&
  	&
  	&
  	&
  	\begin{tabular}[c]{@{}l@{}}None (right)\\ index stable\end{tabular} \\ \hline
  	$C_{18}$ &
  	{\color[HTML]{FE0000} 0-0} &
  	{\color[HTML]{FE0000} 0-0} &
  	{\color[HTML]{FE0000} 0-0} &
  	{\color[HTML]{FE0000} 0-0} &
  	{\color[HTML]{FE0000} 0-0} &
  	1-1 &
  	1-1 &
  	1-1 &
  	&
  	&
  	&
  	&
  	\begin{tabular}[c]{@{}l@{}}None (right)\\ index stable\end{tabular} \\ \hline
  	$C_3 \times S_3$ &
  	{\color[HTML]{FE0000} 0-0} &
  	{\color[HTML]{FE0000} 0-0} &
  	{\color[HTML]{FE0000} 0-0} &
  	{\color[HTML]{FE0000} 0-0} &
  	{\color[HTML]{FE0000} 0-0} &
  	1-0 &
  	1-0 &
  	1-0 &
  	&
  	&
  	&
  	&
  	\begin{tabular}[c]{@{}l@{}}None (right)\\ index stable\end{tabular} \\ \hline
  	$(C_3 \times C_3) : C_2$ &
  	1-1 &
  	1-1 &
  	{\color[HTML]{FE0000} 0-0} &
  	{\color[HTML]{FE0000} 0-0} &
  	{\color[HTML]{FE0000} 0-0} &
  	1-1 &
  	1-0 &
  	1-0 &
  	&
  	&
  	&
  	&
  	\begin{tabular}[c]{@{}l@{}}None (right)\\ index stable\end{tabular} \\ \hline
  	$C_6 \times C_3$ &
  	{\color[HTML]{FE0000} 0-0} &
  	{\color[HTML]{FE0000} 0-0} &
  	{\color[HTML]{FE0000} 0-0} &
  	{\color[HTML]{FE0000} 0-0} &
  	{\color[HTML]{FE0000} 0-0} &
  	1-1 &
  	1-1 &
  	1-1 &
  	&
  	&
  	&
  	&
  	\begin{tabular}[c]{@{}l@{}}None (right)\\ index stable\end{tabular} \\ \hline
  	$C_{19}$ &
  	{\color[HTML]{FE0000} 0-0} &
  	{\color[HTML]{FE0000} 0-0} &
  	{\color[HTML]{FE0000} 0-0} &
  	{\color[HTML]{FE0000} 0-0} &
  	{\color[HTML]{FE0000} 0-0} &
  	1-1 &
  	1-1 &
  	1-1 &
  	&
  	&
  	&
  	&
  	\begin{tabular}[c]{@{}l@{}}None (right)\\ index stable\end{tabular} \\ \hline
  	\begin{tabular}[c]{@{}l@{}}$C_5 : C_4$\\ $( \mbox{Dic}20 )$\end{tabular} &
  	{\color[HTML]{FE0000} 0-0} &
  	{\color[HTML]{FE0000} 0-0} &
  	{\color[HTML]{FE0000} 0-0} &
  	{\color[HTML]{FE0000} 0-0} &
  	{\color[HTML]{FE0000} 0-0} &
  	1-1 &
  	1-1 &
  	1-1 &
  	1-1 &
  	&
  	&
  	&
  	\begin{tabular}[c]{@{}l@{}}None (right)\\ index stable\end{tabular} \\ \hline
  	$C_{20}$ &
  	{\color[HTML]{FE0000} 0-0} &
  	{\color[HTML]{FE0000} 0-0} &
  	{\color[HTML]{FE0000} 0-0} &
  	{\color[HTML]{FE0000} 0-0} &
  	{\color[HTML]{FE0000} 0-0} &
  	1-1 &
  	1-1 &
  	1-1 &
  	1-1 &
  	&
  	&
  	&
  	\begin{tabular}[c]{@{}l@{}}None (right)\\ index stable\end{tabular} \\ \hline
  	\begin{tabular}[c]{@{}l@{}}$C_5 : C_4$\\  $(GA(1,5) )$\end{tabular} &
  	1-1 &
  	{\color[HTML]{FE0000} 0-0} &
  	{\color[HTML]{FE0000} 0-0} &
  	{\color[HTML]{FE0000} 0-0} &
  	{\color[HTML]{FE0000} 0-0} &
  	1-0 &
  	1-0 &
  	1-0 &
  	1-0 &
  	&
  	&
  	&
  	\begin{tabular}[c]{@{}l@{}}None (right)\\ index stable\end{tabular} \\ \hline
  	$D_{20}$ &
  	{\color[HTML]{FE0000} 0-0} &
  	{\color[HTML]{FE0000} 0-0} &
  	{\color[HTML]{FE0000} 0-0} &
  	{\color[HTML]{FE0000} 0-0} &
  	{\color[HTML]{FE0000} 0-0} &
  	1-1 &
  	1-0 &
  	1-0 &
  	1-0 &
  	&
  	&
  	&
  	\begin{tabular}[c]{@{}l@{}}None (right)\\ index stable\end{tabular} \\ \hline
  	$C_{10} \times C_2$ &
  	{\color[HTML]{FE0000} 0-0} &
  	{\color[HTML]{FE0000} 0-0} &
  	{\color[HTML]{FE0000} 0-0} &
  	{\color[HTML]{FE0000} 0-0} &
  	{\color[HTML]{FE0000} 0-0} &
  	1-1 &
  	1-1 &
  	1-1 &
  	1-1 &
  	&
  	&
  	&
  	\begin{tabular}[c]{@{}l@{}}None (right)\\ index stable\end{tabular} \\ \hline
  	$C_7 : C_3$ &
  	1-1 &
  	{\color[HTML]{FE0000} 0-0} &
  	{\color[HTML]{FE0000} 0-0} &
  	{\color[HTML]{FE0000} 0-0} &
  	{\color[HTML]{FE0000} 0-0} &
  	{\color[HTML]{FE0000} 0-0} &
  	1-0 &
  	1-0 &
  	1-1 &
  	&
  	&
  	&
  	\begin{tabular}[c]{@{}l@{}}None (right)\\ index stable\end{tabular} \\ \hline
  	$C_{21}$ &
  	{\color[HTML]{FE0000} 0-0} &
  	{\color[HTML]{FE0000} 0-0} &
  	{\color[HTML]{FE0000} 0-0} &
  	{\color[HTML]{FE0000} 0-0} &
  	{\color[HTML]{FE0000} 0-0} &
  	{\color[HTML]{FE0000} 0-0} &
  	1-1 &
  	1-1 &
  	1-1 &
  	&
  	&
  	&
  	\begin{tabular}[c]{@{}l@{}}None (right)\\ index stable\end{tabular} \\ \hline
  	$D_{22}$ &
  	{\color[HTML]{FE0000} 0-0} &
  	{\color[HTML]{FE0000} 0-0} &
  	{\color[HTML]{FE0000} 0-0} &
  	{\color[HTML]{FE0000} 0-0} &
  	{\color[HTML]{FE0000} 0-0} &
  	{\color[HTML]{FE0000} 0-0} &
  	1-0 &
  	1-0 &
  	1-0 &
  	1-0 &
  	&
  	&
  	\begin{tabular}[c]{@{}l@{}}None (right)\\ index stable\end{tabular} \\ \hline
  	$C_{22}$ &
  	{\color[HTML]{FE0000} 0-0} &
  	{\color[HTML]{FE0000} 0-0} &
  	{\color[HTML]{FE0000} 0-0} &
  	{\color[HTML]{FE0000} 0-0} &
  	{\color[HTML]{FE0000} 0-0} &
  	{\color[HTML]{FE0000} 0-0} &
  	1-1 &
  	1-1 &
  	1-1 &
  	1-1 &
  	&
  	&
  	\begin{tabular}[c]{@{}l@{}}None (right)\\ index stable\end{tabular} \\ \hline
  	$C_{23}$ &
  	{\color[HTML]{FE0000} 0-0} &
  	{\color[HTML]{FE0000} 0-0} &
  	{\color[HTML]{FE0000} 0-0} &
  	{\color[HTML]{FE0000} 0-0} &
  	{\color[HTML]{FE0000} 0-0} &
  	{\color[HTML]{FE0000} 0-0} &
  	1-1 &
  	1-1 &
  	1-1 &
  	1-1 &
  	&
  	&
  	\begin{tabular}[c]{@{}l@{}}None (right)\\ index stable\end{tabular} \\ \hline
  	$C_3 : C_8$ &
  	{\color[HTML]{FE0000} 0-0} &
  	{\color[HTML]{FE0000} 0-0} &
  	{\color[HTML]{FE0000} 0-0} &
  	{\color[HTML]{FE0000} 0-0} &
  	{\color[HTML]{FE0000} 0-0} &
  	{\color[HTML]{FE0000} 0-0} &
  	{\color[HTML]{FE0000} 0-0} &
  	1-0 &
  	1-0 &
  	1-0 &
  	1-1 &
  	&
  	\begin{tabular}[c]{@{}l@{}}None (right)\\ index stable\end{tabular} \\ \hline
  	$C_{24}$ &
  	{\color[HTML]{FE0000} 0-0} &
  	{\color[HTML]{FE0000} 0-0} &
  	{\color[HTML]{FE0000} 0-0} &
  	{\color[HTML]{FE0000} 0-0} &
  	{\color[HTML]{FE0000} 0-0} &
  	{\color[HTML]{FE0000} 0-0} &
  	{\color[HTML]{FE0000} 0-0} &
  	1-1 &
  	1-1 &
  	1-1 &
  	1-1 &
  	&
  	\begin{tabular}[c]{@{}l@{}}None (right)\\ index stable\end{tabular} \\ \hline
  	$SL(2,3)$ &
  	{\color[HTML]{FE0000} 0-0} &
  	{\color[HTML]{FE0000} 0-0} &
  	{\color[HTML]{FE0000} 0-0} &
  	{\color[HTML]{FE0000} 0-0} &
  	{\color[HTML]{FE0000} 0-0} &
  	{\color[HTML]{FE0000} 0-0} &
  	{\color[HTML]{FE0000} 0-0} &
  	1-0 &
  	1-0 &
  	1-0 &
  	1-0 &
  	&
  	\begin{tabular}[c]{@{}l@{}}None (right)\\ index stable\end{tabular} \\ \hline
  	$C_3 : Q_8$ &
  	{\color[HTML]{FE0000} 0-0} &
  	{\color[HTML]{FE0000} 0-0} &
  	{\color[HTML]{FE0000} 0-0} &
  	{\color[HTML]{FE0000} 0-0} &
  	{\color[HTML]{FE0000} 0-0} &
  	{\color[HTML]{FE0000} 0-0} &
  	{\color[HTML]{FE0000} 0-0} &
  	1-0 &
  	1-0 &
  	1-0 &
  	1-0 &
  	&
  	\begin{tabular}[c]{@{}l@{}}None (right)\\ index stable\end{tabular} \\ \hline
  	$C_4 \times S_3$ &
  	{\color[HTML]{FE0000} 0-0} &
  	{\color[HTML]{FE0000} 0-0} &
  	{\color[HTML]{FE0000} 0-0} &
  	{\color[HTML]{FE0000} 0-0} &
  	{\color[HTML]{FE0000} 0-0} &
  	{\color[HTML]{FE0000} 0-0} &
  	{\color[HTML]{FE0000} 0-0} &
  	1-0 &
  	1-0 &
  	1-0 &
  	1-0 &
  	&
  	\begin{tabular}[c]{@{}l@{}}None (right)\\ index stable\end{tabular} \\ \hline
  	$D_{24}$ &
  	{\color[HTML]{FE0000} 0-0} &
  	{\color[HTML]{FE0000} 0-0} &
  	{\color[HTML]{FE0000} 0-0} &
  	{\color[HTML]{FE0000} 0-0} &
  	{\color[HTML]{FE0000} 0-0} &
  	{\color[HTML]{FE0000} 0-0} &
  	{\color[HTML]{FE0000} 0-0} &
  	1-0 &
  	1-0 &
  	1-0 &
  	1-0 &
  	&
  	\begin{tabular}[c]{@{}l@{}}None (right)\\ index stable\end{tabular} \\ \hline
  	$C_2 \times (C_3 : C_4)$ &
  	{\color[HTML]{FE0000} 0-0} &
  	{\color[HTML]{FE0000} 0-0} &
  	{\color[HTML]{FE0000} 0-0} &
  	{\color[HTML]{FE0000} 0-0} &
  	{\color[HTML]{FE0000} 0-0} &
  	{\color[HTML]{FE0000} 0-0} &
  	{\color[HTML]{FE0000} 0-0} &
  	1-0 &
  	1-0 &
  	1-0 &
  	1-0 &
  	&
  	\begin{tabular}[c]{@{}l@{}}None (right)\\ index stable\end{tabular} \\ \hline
  	$(C_6 \times C_2) : C_2$ &
  	{\color[HTML]{FE0000} 0-0} &
  	{\color[HTML]{FE0000} 0-0} &
  	{\color[HTML]{FE0000} 0-0} &
  	{\color[HTML]{FE0000} 0-0} &
  	{\color[HTML]{FE0000} 0-0} &
  	{\color[HTML]{FE0000} 0-0} &
  	{\color[HTML]{FE0000} 0-0} &
  	1-0 &
  	1-0 &
  	1-0 &
  	1-0 &
  	&
  	\begin{tabular}[c]{@{}l@{}}None (right)\\ index stable\end{tabular} \\ \hline
  	$C_{12} \times C_2$ &
  	{\color[HTML]{FE0000} 0-0} &
  	{\color[HTML]{FE0000} 0-0} &
  	{\color[HTML]{FE0000} 0-0} &
  	{\color[HTML]{FE0000} 0-0} &
  	{\color[HTML]{FE0000} 0-0} &
  	{\color[HTML]{FE0000} 0-0} &
  	{\color[HTML]{FE0000} 0-0} &
  	1-1 &
  	1-1 &
  	1-1 &
  	1-1 &
  	&
  	\begin{tabular}[c]{@{}l@{}}None (right)\\ index stable\end{tabular} \\ \hline
  	$C_3 \times D_8$ &
  	{\color[HTML]{FE0000} 0-0} &
  	{\color[HTML]{FE0000} 0-0} &
  	{\color[HTML]{FE0000} 0-0} &
  	{\color[HTML]{FE0000} 0-0} &
  	{\color[HTML]{FE0000} 0-0} &
  	{\color[HTML]{FE0000} 0-0} &
  	{\color[HTML]{FE0000} 0-0} &
  	1-0 &
  	1-0 &
  	1-0 &
  	1-0 &
  	&
  	\begin{tabular}[c]{@{}l@{}}None (right)\\ index stable\end{tabular} \\ \hline
  	$C_3 \times Q_8$ &
  	{\color[HTML]{FE0000} 0-0} &
  	{\color[HTML]{FE0000} 0-0} &
  	{\color[HTML]{FE0000} 0-0} &
  	{\color[HTML]{FE0000} 0-0} &
  	{\color[HTML]{FE0000} 0-0} &
  	{\color[HTML]{FE0000} 0-0} &
  	{\color[HTML]{FE0000} 0-0} &
  	1-0 &
  	1-0 &
  	1-1 &
  	1-1 &
  	&
  	\begin{tabular}[c]{@{}l@{}}None (right)\\ index stable\end{tabular} \\ \hline
  	$S_4$ &
  	1-1 &
  	{\color[HTML]{FE0000} 0-0} &
  	{\color[HTML]{FE0000} 0-0} &
  	{\color[HTML]{FE0000} 0-0} &
  	{\color[HTML]{FE0000} 0-0} &
  	{\color[HTML]{FE0000} 0-0} &
  	{\color[HTML]{FE0000} 0-0} &
  	1-0 &
  	1-0 &
  	1-0 &
  	1-0 &
  	&
  	\begin{tabular}[c]{@{}l@{}}None (right)\\ index stable\end{tabular} \\ \hline
  	$C_2 \times A_4$ &
  	{\color[HTML]{FE0000} 0-0} &
  	{\color[HTML]{FE0000} 0-0} &
  	{\color[HTML]{FE0000} 0-0} &
  	{\color[HTML]{FE0000} 0-0} &
  	{\color[HTML]{FE0000} 0-0} &
  	{\color[HTML]{FE0000} 0-0} &
  	{\color[HTML]{FE0000} 0-0} &
  	1-0 &
  	1-0 &
  	1-0 &
  	1-0 &
  	&
  	\begin{tabular}[c]{@{}l@{}}None (right)\\ index stable\end{tabular} \\ \hline
  	$C_2 \times C_2 \times S_3$ &
  	{\color[HTML]{FE0000} 0-0} &
  	{\color[HTML]{FE0000} 0-0} &
  	{\color[HTML]{FE0000} 0-0} &
  	{\color[HTML]{FE0000} 0-0} &
  	{\color[HTML]{FE0000} 0-0} &
  	{\color[HTML]{FE0000} 0-0} &
  	{\color[HTML]{FE0000} 0-0} &
  	1-0 &
  	1-0 &
  	1-0 &
  	1-0 &
  	&
  	\begin{tabular}[c]{@{}l@{}}None (right)\\ index stable\end{tabular} \\ \hline
  	$C_6 \times C_2 \times C_2$ &
  	{\color[HTML]{FE0000} 0-0} &
  	{\color[HTML]{FE0000} 0-0} &
  	{\color[HTML]{FE0000} 0-0} &
  	{\color[HTML]{FE0000} 0-0} &
  	{\color[HTML]{FE0000} 0-0} &
  	{\color[HTML]{FE0000} 0-0} &
  	{\color[HTML]{FE0000} 0-0} &
  	1-1 &
  	1-1 &
  	1-1 &
  	1-1 &
  	&
  	\begin{tabular}[c]{@{}l@{}}None (right)\\ index stable\end{tabular} \\ \hline
  	$C_{25}$ &
  	{\color[HTML]{FE0000} 0-0} &
  	{\color[HTML]{FE0000} 0-0} &
  	{\color[HTML]{FE0000} 0-0} &
  	{\color[HTML]{FE0000} 0-0} &
  	{\color[HTML]{FE0000} 0-0} &
  	{\color[HTML]{FE0000} 0-0} &
  	{\color[HTML]{FE0000} 0-0} &
  	1-1 &
  	1-1 &
  	1-1 &
  	1-1 &
  	&
  	\begin{tabular}[c]{@{}l@{}}None (right)\\ index stable\end{tabular} \\ \hline
  	$C_5 \times C_5$ &
  	1-1 &
  	{\color[HTML]{FE0000} 0-0} &
  	{\color[HTML]{FE0000} 0-0} &
  	{\color[HTML]{FE0000} 0-0} &
  	{\color[HTML]{FE0000} 0-0} &
  	1-1 &
  	1-1 &
  	1-1 &
  	1-1 &
  	1-1 &
  	1-1 &
  	&
  	\begin{tabular}[c]{@{}l@{}}None (right)\\ index stable\end{tabular} \\ \hline
  	$D_{26}$ &
  	{\color[HTML]{FE0000} 0-0} &
  	{\color[HTML]{FE0000} 0-0} &
  	{\color[HTML]{FE0000} 0-0} &
  	{\color[HTML]{FE0000} 0-0} &
  	{\color[HTML]{FE0000} 0-0} &
  	{\color[HTML]{FE0000} 0-0} &
  	1-0 &
  	1-0 &
  	1-0 &
  	1-0 &
  	1-{\color[HTML]{FE0000}0} &
  	1-{\color[HTML]{FE0000}0}&
  	\begin{tabular}[c]{@{}l@{}}None (right)\\ index stable\end{tabular} \\ \hline
  	$C_{26}$ &
  	{\color[HTML]{FE0000} 0-0} &
  	{\color[HTML]{FE0000} 0-0} &
  	{\color[HTML]{FE0000} 0-0} &
  	{\color[HTML]{FE0000} 0-0} &
  	{\color[HTML]{FE0000} 0-0} &
  	{\color[HTML]{FE0000} 0-0} &
  	{\color[HTML]{FE0000} 0-0} &
  	1-1&
  	1-1&
  	1-1&
  	1-1&
  	1-1&
  	\begin{tabular}[c]{@{}l@{}}None (right)\\ index stable\end{tabular} \\ \hline
  	$C_{27}$ &
  	{\color[HTML]{FE0000} 0-0} &
  	{\color[HTML]{FE0000} 0-0} &
  	{\color[HTML]{FE0000} 0-0} &
  	{\color[HTML]{FE0000} 0-0} &
  	{\color[HTML]{FE0000} 0-0} &
  	{\color[HTML]{FE0000} 0-0} &
  	{\color[HTML]{FE0000} 0-0} &
  	{\color[HTML]{999999}1-1}&
  	1-1&
  	1-1&
  	1-1&
  	1-1&
  	\begin{tabular}[c]{@{}l@{}}None (right)\\ index stable\end{tabular} \\ \hline
  	$C_9 \times C_3$ &
  	{\color[HTML]{FE0000} 0-0} &
  	{\color[HTML]{FE0000} 0-0} &
  	{\color[HTML]{FE0000} 0-0} &
  	{\color[HTML]{FE0000} 0-0} &
  	{\color[HTML]{FE0000} 0-0} &
  	{\color[HTML]{FE0000} 0-0} &
  	{\color[HTML]{FE0000} 0-0} &
  	{\color[HTML]{FE0000} 0-0} &
  	1-1&
  	1-1&
  	1-1&
  	1-1&
  	\begin{tabular}[c]{@{}l@{}}None (right)\\ index stable\end{tabular} \\ \hline
  	$(C_3 \times C_3) : C_3$ &
  	1-1 &
  	{\color[HTML]{FE0000} 0-0} &
  	{\color[HTML]{FE0000} 0-0} &
  	{\color[HTML]{FE0000} 0-0} &
  	{\color[HTML]{FE0000} 0-0} &
  	{\color[HTML]{999999}1}-{\color[HTML]{FE0000}0} &
  	{\color[HTML]{999999}1}-{\color[HTML]{FE0000}0} &
  	{\color[HTML]{999999}1}-{\color[HTML]{FE0000}0} &
  	1-{\color[HTML]{999999}1}&
  	1-{\color[HTML]{999999}1}&
  	1-{\color[HTML]{999999}1}&
  	1-{\color[HTML]{999999}1}&
  	\begin{tabular}[c]{@{}l@{}}None (right)\\ index stable\end{tabular} \\ \hline
  	$C_9 : C_3$ &
  	{\color[HTML]{FE0000} 0-0} &
  	{\color[HTML]{FE0000} 0-0} &
  	{\color[HTML]{FE0000} 0-0} &
  	{\color[HTML]{FE0000} 0-0} &
  	{\color[HTML]{FE0000} 0-0} &
  	{\color[HTML]{FE0000} 0-0} &
  	{\color[HTML]{FE0000} 0-0} &
  	{\color[HTML]{FE0000} 0-0}&
  	1-{\color[HTML]{FE0000}0}&
  	1-{\color[HTML]{FE0000}0}&
  	1-{\color[HTML]{FE0000}0}&
  	1-{\color[HTML]{999999}1}&
  	\begin{tabular}[c]{@{}l@{}}None (right)\\ index stable\end{tabular} \\ \hline
  	$C_3 \times C_3 \times C_3$ &
  	1-1 &
  	1-1 &
  	{\color[HTML]{FE0000} 0-0} &
  	1-1 &
  	1-1 &
  	1-1 &
  	1-1 &
  	1-1 &
  	1-1&
  	1-1&
  	1-1&
  	1-1&
  	\begin{tabular}[c]{@{}l@{}}None (right)\\ index stable\end{tabular} \\ \hline
  \end{longtable}
\section{Completely determination of all finite index stable groups: main theorem}
Now, we are ready to completely characterize finite index stable groups by six steps.
This solves one of the main open problems of the theory.
\begin{theorem}
\label{finitegroups}
There are only $14$  finite $($right$)$ index stable  groups as follows:\\
$$
C_1,\;
 C_2,\;   C_2\times C_2,\; C_2\times C_2\times C_2,\; C_2\times C_2\times C_2\times C_2,
  $$
  $$
 C_3,\; C_3\times C_3,\;
 C_4,\;  C_4\times C_2,\;
   C_5,\;
       C_7,\;
       S_3,\;
   D_8,\; Q_8
$$
$($i.e., all groups of orders $\leq 9$ except $C_6$, $C_8$, $C_9$ together with $C_2\times C_2\times C_2\times C_2)$.
\end{theorem}
\begin{proof}
First note that all finite (left, right) index stable groups are among groups $G$ with order $|G| = 2^q\cdot 3^r \cdot 5^s\cdot 7^t$
for non-negative integers $q, r, s, t$, by Theorem 3.12 of \cite{MH3}.
We prove this theorem by six steps. During the proof, also we use the important fact from Corollary 3.23 of \cite{MH2} (property (c)
of Subsection \ref{indsubsets}), repeatedly.\\
\textbf{Step 1.} The groups $(C_2)^5$, $(C_3)^3$, $(C_5)^2$ and $(C_7)^2$ are the first powers of
$C_2, C_3, C_5$ and $C_7$ that are not index stable.
From Table 1,2, it is evident that groups of this form with smaller orders are index stable.
Additionally, Table 1,2 shows that $C_3 \times C_3\times C_3$ and $C_5 \times C_5$ are not index stable.\\
On the other hand, $C_2\times C_2\times C_2\times C_2 \times C_2$ is not index stable.
For if
$$ A:= \{00000, 10000, 01000, 00100, 00010, 00001\}\subseteq (C_2)^5,$$
i.e., the set
of elements with at most one coordinate equal to 1.  Then, for any
$b\in (C_2)^5$,  the elements of  $A+b$  will be those which differ from  $b$  in at
most one coordinate; hence for elements  $b, b'$,  the sets  $A+b$  and
$A+b'$ will be disjoint if and only if  $b$  and  $b'$  differ in at least
3 coordinates.
So let  $B_1 = \{00000, 11100, 00111, 11011\}$,  and
$B_2 = \{00000, 11111\}$.
It is easy to check that every element of  $G$  agrees in at least 3
coordinates with an element of  $B_2$,  hence disagrees with such an
element in at most 2 coordinates; so no additional elements can be
added to  $B_2$  and keep its product with  $A$  direct; so  $B_2$  is a
sub-factor related to $ A$.  Similarly,  $B_1$  is also such a sub-factor,
and so  $A$  is not index stable.\\
Also, applying the GAP code \href{https://github.com/momoeysfn/Subindices/blob/main/lib/rbsfrandom.g}{(link to code)} for $A=\{00 , 01 , 10 \}\subseteq C_7\times C_7$
shows that $$B_1=\{00, 02, 22, 24, 26, 30, 41, 44, 52, 65\}$$ of size 10, and
$$B_2 =\{00, 04, 11, 15, 22, 33, 36, 40, 44, 51, 55, 62, 66\}$$ of size 13 are two sub-factors of
$C_7\times C_7$ relative to $A$.\\
Therefore,  $(C_2)^5$, $(C_3)^3$, $(C_5)^2$ and $(C_7)^2$ are the first powers of
$C_2, C_3, C_5$ and $C_7$ that are not index stable.
 Because in subsequent powers, we will have a subgroup that is isomorphic to these groups, therefore, they are not index stable.
 \\
\textbf{Step 2.} All finite right index stable groups $G$ are of the order
\begin{align*}
	|G| &\in \{2^q \cdot 3^r \cdot 5^s \cdot 7^t : q=0,1,2,3,4, \; r=0,1,2, \; s=0,1, \; t=0,1\} \\
	&= \{1, 2, 3, 4, 5, 6, 7, 8, 9, 10, 12, 14, 15, 16, 18, \ldots, 1680, 2520, 5040\}
\end{align*}
Because for larger powers, we will have a subgroup of order $32$, $27$, $25$  or $49$. It is clear from Table 1,2 that there is no index stable group of order $25$ and $27$. For the order of $49$, we have two groups, $C_7 \times C_7$ and $C_{49}$, both of which are not index stable. For the order of $32$, after considering the proper subgroups of these groups and examining Table 1,2, only the group ${(C_2)}^5$ has the property that all of its proper subgroups are index stable. However, as we have seen, that group itself is not index stable. \\
\textbf{Step 3.} Right index stability of groups $G$ of orders $1 < |G| \leq 2^4$: from Table 1,2, it can be deduced that only all groups with orders $\leq 9$, except $C_6$, $C_8$, $C_9$, along with $C_2 \times C_2 \times C_2 \times C_2$, are index stable.
\\
\textbf{Step 4.} Index stability of abelian groups $G$ of orders $|G| > 2^4$: such an index stable group does not exist; according to Theorem 3.12 from \cite{MH3}, it must be a $p$-group with the specified order. If $p$ is $3$, 5, or $7$, and $|G| > 2^4$, they do not conform to the form mentioned earlier. In the case of a $2$-group, considering the abelian decomposition, it can only be a power of $C_2$; otherwise, it does not possess the property that all of its proper subgroups are index stable. This is evident from Table 1,2, where $C_8$, $C_4 \times C_4$, and $C_2 \times C_2 \times C_4$ are not index stable, and consequently, groups containing subgroups of this form are also not index stable. Furthermore, if $|G| > 2^4$ and it is a power of $C_2$, it must include $C_2 \times C_2 \times C_2 \times C_2 \times C_2$, which implies it is not index stable.  \\
For continuation of the proof, using the code \href{https://github.com/momoeysfn/Subindices/blob/main/examples/IndexStability.g}{(link to code)}, we consider non-abelian groups of the specified order form within intervals of powers of 2. For each group $G$ with an order between $2^4$ and $2^5$, all of its proper subgroups have orders ranging from 1 to $2^4$, and the index stability of these subgroups is known. In each step of this procedure we examine the index stability of groups that possess the property of having all of their proper subgroups index stable. This process helps us determine all index stable groups up to order $2^5$. In this manner, an inductive approach allows us to determine the status of groups from $2^n$ to $2^{n+1}$ once it is established up to $2^n$. Any special cases with all subgroups being index stable are noted below and if such cases do not exist, we have moved on from that interval. Since the orders must adhere to the form mentioned in step two, this process is finite.\\
\textbf{Step 5.} Right index stability of non-abelian groups $G$ of order $2^4 < |G| \leq 1680$:  we have two cases :\\
\textbf{Case 1.} Right index stability of $G$ with $2^4 < |G| \leq 2^7$: only
$C_2^3:C_7$ and $C_2^4:C_5$ have the property that
all their proper subgroups are index stable. For these two groups we have examples of non right index stable
subsets as follows \href{https://github.com/momoeysfn/Subindices/blob/main/lib/th31exmpls.md}{(link to recorded output)}.\\
\textbf{Case 2.}  Right index stability of $G$ with $2^7 < |G| \leq 1680$:
all such groups $G$ contain a none right index stable subgroup obtained from the previous steps.
Hence, there are no index stable groups in this case. \\
\textbf{Step 6.} The remains groups are those whose orders are $2520=2^3\cdot3^2\cdot5\cdot7=\frac{7!}{2}$ or
$5040=2^4\cdot3^2\cdot5\cdot7=7!$  which GAP does not support them. Fortunately, we can give a
theoretical proof for their non-index stability as follows.
A Sylow 3-subgroup of such groups would either be normal, or self-normalizing, in which case
there would be a normal 3-complement by Burnside's Transfer Theorem. Therefore they contain some none right
index stable subgroup, and so the proof is complete.
\end{proof}
\begin{cor} \label{rlfinite}
For a finite group $G$ the followings are equivalent:\\
$($a$)$ $G$ is index stable;\\
$($b$)$ $G$ is right index stable;\\
$($c$)$ $G$ is left index stable;\\
$($d$)$ $G$ is one of the 14 groups mentioned in Theorem 3.1 $($up to isomorphism$)$.\\
\end{cor}

\section{Answers and solutions to some other questions, open problems, and conjectures}
Since the theory of sub-factors and sub-indices is completely new, it is natural that many questions,
 open problems, conjectures and research projects are raised about it.
 A number of them have been mentioned in the \cite{MH2,MH3} that
 some of them are fundamental and have special importance. So far we
  have answered few of them in whole or in part.
   In order not to miss any important items about finite groups, we will list and explain one by one.\\
\subsection{Open problems and questions from \cite{MH2}}
Below, all the questions and problems raised from the first article, which is the beginning of
the theory, will be discussed in the same order and number mentioned in that paper. In some cases, we have used various methods and heuristics to search for counterexamples. If none were found, we reported the exact order up to which we confirmed their absence.\\
{\bf Question I.} Are all products of $\mathbb{Z}_2$ index stable?
What about $\mathbb{Z}_3$, $\mathbb{Z}_4$, $\mathbb{Z}_5$ and $\mathbb{Z}_7$?\\
For finite case, this question has been answered in Theorem \ref{finitegroups}. But in general, it is still open.\\
{\bf Problem II.} Let $k\geq 2$ be a given natural number.\\
{\bf (a)} Characterize all $n$ such that $\mathbb{Z}_n$ (resp. $S_n$, $A_n$, $D_{2n}$, etc.) is $k$-index stable. \\
{\bf (b)} Characterize or classify all finite groups $G$ of order $m$ such that it is $k$-index stable,
where $m$ is a fixed integer and $2\leq k\leq \lfloor\frac{m}{2}\rfloor$ (e.g., $k=2$ and $m=16$).\\
This is still open. Maybe the problem can be solved similar to what we did in Theorem \ref{finitegroups}. Regarding part (a) we will mention a conjecture after Theorem \ref{41}.\\
{\bf Problems and questions III.}\\
{\bf (a)} Characterize all $n$ such that $S_n$ (resp. $A_n$, $D_{2n}$, etc.) is index stable. \\
All the index stable cases are $S_1$, $S_2$, $S_3$, $A_2$, $A_3$, $D_2$, $D_4$, $D_6$, $D_8$.\\
{\bf (b)} Characterize or classify all index stable subgroups of $(\mathbb{R},+)$ and $(\mathbb{C},+)$.
Especially, is  the additive group of rational (resp. real, complex) numbers index stable?\\
This is still open.\\
{\bf (c)} Give some subsets of some (finite, infinite) groups such that all its sub-indices are different (i.e.,
it has no any type of index stability). \\
There is no any counterexample in groups up to order 27, see \href{https://github.com/momoeysfn/Subindices/blob/main/examples/QsIIIc.g}{link}.\\
{\bf (d)} Give finite and infinite examples of a group that is right (resp. left) but not left (resp. right) index stable:\\
There is no any example, it has been proved that right and left index stabilities are equivalent (see \cite{MH3}, the eighth page).\\
Also, give finite and infinite examples of groups that are both left and right index
stable but not index stable:\\
For finite groups they are equivalent, by Corollary \ref{rlfinite}. But, it is still open for the infinite case.
\\
{\bf (h)} If $Dif_\ell^2(A)=G$, $Dif_\ell(A)\neq G$ (the second condition is lost in the original version) and $A$ is right index stable, then $|G:A|_r=2$,
and vice versa (analogously for the left and two-sided cases).\\
Counterexample: $G = C_3 \times C_3,\ A=\{00,01,10\}$ \href{https://github.com/momoeysfn/Subindices/blob/main/examples/QsIIIh.g}{link}, and for the converse $G = C_2 , \ A=\{0\}$. \\
{\bf (i)} If $A\subseteq H\leq G$ and $A$ is index stable in $G$, then it is so
in $H$ (and vice versa).\\
If $|G:H|$ is finite (e.g., if $G$ is finite), then index stability of $A$ in $G$ implies $A$ is index stable in $H$,
for the converse, it is still open (for both finite and infinite case). Of course,  there is no any counterexample in groups up to order 18, see \href{https://github.com/momoeysfn/Subindices/blob/main/examples/QsIIIi.g}{link}. \\
{\bf Problem IV.} Prove or disprove: \\
(a) If $A_1\subseteq A_2$ then $|G:A_2|^-\leq |G:A_1|^-$
or $|G:A_2|^+\leq |G:A_1|^-$;\\
Counterexample: For the first  $G = C_{12}, \ A_2=\{0,1,5,6\}, \ A_1=\{0,1,6\}$, and the
second  $G = C_6, \ A_1=A_2=\{0,1\}$ \href{https://github.com/momoeysfn/Subindices/blob/main/examples/QsIVa.g}{(link)}.
\\
(b) If $A$ is infinite then $|G:Dif_\ell(A)|^+\leq |G:A|^-$.\\
It is still open.
\\
{\bf Question V.} Is it true that $\mbox{SubF}_r(A)=\mbox{SubF}_\ell(A)$ if
$A$ is symmetric and vice versa? \\
Counterexample: For the first part, $G=S_3$, $A=\{ (), (2,3) \} $, and for the converse $G=C_3$, $A=\{0,1\} $ \href{https://github.com/momoeysfn/Subindices/blob/main/examples/QsV.g}{(link)} (of course, we have  $|G:A|^\pm=|G:A^{-1}|^\pm=|G:A|_\pm$ if $A$ is symmetric).
\subsection{Open problems and questions from \cite{MH3}}
For this part, we do the same as what we did in the previous subsection. \\
{\bf Problem I.} Prove or disprove:\\
(1) If $|A|> \sqrt{|G|}$, then $|G:A|=1$ (the converse is valid).\\
Counterexample: $G = D_{10}, A=\{1,a,b,a^4.b\}, |G:A|=2$ \href{https://github.com/momoeysfn/Subindices/blob/main/examples/QsI1.g}{(link)}.
\\
(2) If $A$ is right and left index stable with sub-indices 1 or 2, then $|A|>\frac{|G|}{3}$.\\
Counterexample: $G = D_{10}, \ A=\{1,a,a^2\}, \ |G:A|=2$ \href{https://github.com/momoeysfn/Subindices/blob/main/examples/QsI2.g}{(link)}.
\\
(3) If $G$ is abelian and $A$ is index stable with index 1 or 2, then $|A|>\frac{|G|}{3}$.\\
Counterexample: $G = C_{10}, \ A=\{0,1,3\}, \ |G:A|=2$ \href{https://github.com/momoeysfn/Subindices/blob/main/examples/QsI3.g}{(link)}.
\\
(4) There is a finite group $G$ with a subset $A$ such that $|A|=\lceil \sqrt{|G|}\rceil$ and $|G:A|=1$.\\
True; consider $G =	C_{10}$, $\ A=\{0,1,2,3\}$. \href{https://github.com/momoeysfn/Subindices/blob/main/examples/QsI4.g}{(link)}\\
Hence, we introduce another question here.\\
{\bf New Question.} Is there a finite group $G$ with a subset $A$ such that $|A|=\lfloor \sqrt{|G|}\rfloor$  and $|G:A|=1$?
\\
The answer is negative for all groups of orders  up to 23 \href{https://github.com/momoeysfn/Subindices/blob/main/examples/QsI4.g}{(link)}.
\\
{\bf Problem II.} Determine or classify all subsets $A$ of a finite group $G$ such that one (some) of
the following equalities holds:\\
$|G:A|^+=\lfloor \frac{|G|}{|A|}\rfloor$, $|G:A|^-=\lceil\frac{|G|}{|\Dif_\ell(A)|}\rceil$,
$|G:A|_r=\lfloor \frac{|G|}{|A|}\rfloor$, $|G:A|_r=\lceil\frac{|G|}{|\Dif_\ell(A)|}\rceil$,
$|G:A|_r=\lfloor \frac{|G|}{|A|}\rfloor=\lceil\frac{|G|}{|\Dif_\ell(A)|}\rceil$, etc.\\
This is a research project.\\
{\bf Problem III.} Determine all finite (right) index stable groups.\\
This is completely solved by Theorem \ref{finitegroups}. \\
{\bf Question IV.} Do we have $|G:\{1,a\}|^-=|G:\{1,a\}|_-$, for every $a\in G$? More generally, if
$\Dif_\ell(A)=\Dif_r(A)$ then is it true that $|G:A|^{\pm}=|G:A|_{\pm}$?\\
There is no counterexample up to order 23, for the first part. Also, no counterexample is found for
the second part up to order 18 \href{https://github.com/momoeysfn/Subindices/blob/main/examples/QsIV2.g}{(link)}. \\
{\bf Conjecture V.} Every right sub-factor of $G$ related to $A$ can be gotten from the above algorithm (i.e., Theorem 4.1 \cite{MH3}).\\
This conjecture is true by Theorem \ref{rbsfproof}.\\
{\bf Problem VI.} Find an algorithm (in finite groups) for obtaining a right sub-factor of $G$ related to $A$
with the most (resp. least) size.\\
This is still open.\\
{\bf Question VII.} Is it true that $|\mathbb{Z}_n:\{0,1\}|^-=\lceil \frac{n}{3}\rceil$, for every $n\geq 2$?\\
Yes, due to the next theorem.
\begin{theorem}
\label{41}
The subset $\{0,1\}$ takes its $($relative$)$ maximum upper and minimum lower index in $\mathbb{Z}_n$,
for all $n\geq 2$ $($also see the next remark$)$, i.e.,
$$|\mathbb{Z}_n:\{0,1\}|^-=\lceil\frac{n}{3}\rceil \; , \; |\mathbb{Z}_n:\{0,1\}|^+=\lfloor\frac{n}{2}\rfloor $$
\end{theorem}
\begin{proof}
Note that $|\{0,1\}|=2$, $|\{0,1\}-\{0,1\}|=|\{0,1,n-1\}|=3$ and
$$
\lceil \frac{n}{3}\rceil\leq |\mathbb{Z}_n:\{0,1\}|^-\leq |\mathbb{Z}_n:\{0,1\}|^+\leq \lfloor \frac{n}{2}\rfloor.
$$
Putting $B:=\left\{0,2,\ldots,2(\lfloor n/2\rfloor - 1)\right\}$ and
$$B':=\left\{\begin{array}{cc}\left\{0,3,\ldots,3(\lceil\frac{n}{3}\rceil- 1)\right\} \; ; &\; n\not\equiv 1 \pmod 3 \\
\left\{0,3,\ldots,3(\lceil\frac{n}{3}\rceil- 2),3\lceil\frac{n}{3}\rceil-4\right\} \; ;& \; n\equiv 1 \pmod 3
\end{array}
\right.$$
 we have $|B|=\lfloor \frac{n}{2}\rfloor$
and $|B'|=\lceil \frac{n}{3}\rceil$. Due to the above inequality, it is enough to show that the summations
$B+\{0,1\}$ and $B'+\{0,1\}$ are direct, or equivalently
$$(B-B)\cap \{0,1,n-1\}=\{0\}=(B'-B')\cap \{0,1,n-1\}.$$
Let $x,y$ are both elements of $B$ or $B'$ and $y\neq 0$.  If $x+(n-y)=n-1$, then we get a contradiction,
since $2|x-y$ or $3|x-y$, or one of the relations  $3|x$, $y=n-2$ or $3|y$, $x=n-2$ be occurred.\\
Hence, suppose that $x+(n-y)=1$ and consider the following cases:\\
(i) $x,y\in B$: we conclude that $n$ is odd,
$0\leq x\leq n-3$, $2\leq y\leq n-3$, and so $5-n\leq y-x\leq n-3$ that is impossible. \\
(ii) $x,y\in B'$ and $n\not\equiv 1 \pmod 3$: the equality $y-x=n-1$ gives a contradiction, since $3|x-y$.   \\
(iii) $x,y\in B'$ and $n\equiv 1 \pmod 3$: then both $x$ and $y$ must be divided by 3 and so $n-1\leq y-x\leq n-4$
that is a contradiction.\\
Finally, note that if $y=0$, then one can see $x-y\neq 1, n-1$. Therefore, the proof is complete.
\end{proof}
Regarding problem II(a) of subsection 4.1, the above theorem, considering Table 1,2 we have an important conjecture for finite cyclic groups as follows.\\
{\bf New conjecture.}
If $n>11$ then $\mathbb{Z}_n$ is not $k$-index stable if and only if $2\leq k \leq \lfloor \frac{n}{3} \rfloor$.

\begin{rem}
Since the set of all solutions of the equation  $\lceil \frac{n}{3}\rceil=\lfloor \frac{n}{2}\rfloor$ is $\{2,3,4,5,7\}$,
we deduce that $\mathbb{Z}_n$ ($n\geq 2$) is 2-index stable if and only if $n\in \{2,3,4,5,7\}$ which agrees with
Lemma 3.17 of \cite{MH2}. But, if $n\notin \{1, 2,3,4,5,7\}$ then $\{0,1\}$ takes its relatively least (resp. largest) possible upper
(resp. lower) index which means it is relatively strong index unstable. Note that a subset $A$ of a finite group $G$ is
relatively (resp. absolutely)  strong right index unstable
if $\frac{|G|}{|\Dif_\ell(A)|}\rceil=|G:A|^-<|G:A|^+= \lfloor \frac{|G|}{|A|}\rfloor$ (resp. $2=|G:A|^-<|G:A|^+=\lfloor \frac{|G|}{|2|}\rfloor$).
\end{rem}
{\bf GAP Project VIII.} Let $G$ be a finite group and $A\subseteq G$.
Give some GAP codes for evaluating or checking the following items one time by the definitions and another time
by using the related properties or algorithms:\\
(a) $\SubF_\ell(A)$ and $\SubF_r(A)$ for a given subset $A$;\\
(b) All sub-indices of $A$;\\
(c) The six types of index stability of $G$, and the set of all integers $\alpha$
such that $G$ is $\alpha$-index stable.\\
(d) Also, some codes for checking the mentioned problems and questions.\\
These should be helpful for answering questions, solving problems or getting some counterexamples.\\
\underline{This project} has been carried out to a considerable extent in this paper.

\subsection*{Acknowledgement} The authors are grateful to Professor George Bergman for his valuable comments and providing a theoretical method to prove the non-index stability of $G = (\mathbb{Z}_2)^5$.

\end{document}